\documentclass{amsart}
\usepackage{latexsym,amsxtra,amscd,ifthen}
\usepackage{amsfonts}
\usepackage{verbatim}
\usepackage{amsmath}
\usepackage{amsthm}
\usepackage{amssymb}
\usepackage{url}
\usepackage[arrow,matrix]{xy}

\allowdisplaybreaks[4]

\newtheorem{theorem}{Theorem}[section]

\newtheorem{corollary}[theorem]{Corollary}
\newtheorem{lemma}[theorem]{Lemma}

\newtheorem{proposition}[theorem]{Proposition}
\newtheorem{prop-def}[theorem]{Proposition-Definition}
\theoremstyle{definition}
\newtheorem{definition}[theorem]{Definition}

\newtheorem{example}[theorem]{Example}

\numberwithin{equation}{section}

\def\deg{{\rm deg}}

\def\1{\mathbbold{1}}

\def\fm{\mathfrak m}

\def\GKdim{{\rm GKdim}\,}

\begin{document}
\title[Graded Poisson Hopf algebras]
{Unimodular graded Poisson Hopf algebras}

\author{K.A. Brown}
\address{School of Mathematics and Statistics\\
University of Glasgow\\ Glasgow G12 8QW\\
Scotland}
\email{ken.brown@glasgow.ac.uk}

\author{J.J. Zhang}
\address{Department of Mathematics\\
Box 354350\\
University of Washington\\
Seattle, WA 98195, USA}
\email{zhang@math.washington.edu}



\begin{abstract}
Let $A$ be a Poisson Hopf algebra over an algebraically closed 
field of characteristic zero. If $A$ is
finitely generated and connected graded as an algebra
and its Poisson bracket is homogeneous of degree
$d\geq 0$, then $A$ is unimodular; that is, the modular
derivation of $A$ is zero. This is a Poisson analogue of a 
recent result concerning Hopf algebras which are connected 
graded as algebras.
\end{abstract}

\maketitle

\dedicatory{}%
\commby{}%

\setcounter{section}{-1}
\section{Introduction}
\label{zzsec0}

\bigskip

Poisson algebras have lately been playing an important role in algebra,
geometry, mathematical physics and other subjects. 
For example,  Poisson structures have been used in the study of discriminants,
in the work of Nguyen-Trampel-Yakimov \cite{NTY} and 
Levitt-Yakimov \cite{LY}, and in the representation theory of 
Sklyanin algebras of global dimension three, in the work of 
Walton-Wang-Yakimov \cite{WWY}. Restricted Poisson Hopf algebras were introduced
in \cite{BK}, and further investigated in \cite{BYZ}. 

Throughout, $\Bbbk$ will denote an algebraically 
closed field of characteristic zero; all algebras considered 
in what follows are $\Bbbk$-algebras, and all unadorned 
tensor products are over $\Bbbk$.

We are concerned here with the unimodularity of a Poisson Hopf algebra 
$A$ which is an affine connected graded algebra. Here, in the light of 
the smoothness which holds in our characteristic zero setting, $A$ is a 
polynomial algebra in finitely many variables. Since Poisson Hopf algebras
were introduced by Drinfeld \cite{Dr1} in 1985, (see Definition \ref{zzdef1.1}), they have
been intensively studied in connection with homological algebra and
deformation quantization; see, for example, recent work in \cite{LWZ1, LWZ2, 
LWW, LW}. A Poisson algebra is said to be \emph{unimodular} if the class of its 
modular derivation in a certain factor of the first Poisson cohomology group is 
trivial; further details and references are given in section 3.1. The purpose 
of this paper is to provide further evidence that unimodularity for 
commutative Poisson algebrs is an analogue of the Calabi-Yau property of 
a noncommutative algebra, reinforcing results in this direction in 
\cite{Do, LWW, LWZ2}. Thus, just as is the case for the Calabi-Yau property in a
noncommutative setting, unimodularity can be viewed as a homological
property of Poisson algebras.

A result of \cite{BGZ} states that a Hopf $\Bbbk$-algebra of  finite 
Gelfand-Kirillov dimension which is 
connected graded as an algebra  is Calabi-Yau. One might 
suspect that there should be a Poisson version of this 
result, and indeed our main result is
the following theorem, whose proof uses this noncommutative result from \cite{BGZ}, 
applied to the Poisson enveloping algebra of a graded Poisson Hopf algebra.

\begin{theorem}
\label{zzthm0.1}
Let $A$ be a Poisson Hopf $\Bbbk-$algebra. Suppose
that
\begin{enumerate}
\item[(i)]
as an algebra, $A$ is finitely generated and connected graded; and that
\item[(ii)]
the Poisson bracket $\{-,-\}$ is homogeneous of degree $d\geq 0$.
\end{enumerate}
Then $A$ is unimodular; that is, the modular derivation of $A$ is zero.
\end{theorem}

In parallel with duality  results for Hochschild homology and 
cohomology in the setting of noncommutative Calabi-Yau algebras, 
the above result yields consequences for duality of Poisson 
homology and cohomology. Thus, as an immediate consequence of the theorem and 
\cite[Theorem 3.5 and Remark 3.6]{LWW}, we have

\begin{corollary}
\label{zzcor0.2}
Let $A$ be as in the above theorem and let $n=\GKdim A$, the rank of 
the polynomial algebra $A$. Then there is a 
Poincar{\'e} duality between the Poisson homology and cohomology:
for every Poisson $A$-module $M$, there is a functorial isomorphism 
$$HP_i(M)\cong HP^{n-i}(M)$$
for all $i = 0, \ldots , n$.
\end{corollary}

We refer to \cite{LWW, LWZ2} for other terminology in Corollary \ref{zzcor0.2}. 
Section 1 contains background on Poisson Hopf algebras, algebra, coalgebra 
and Hopf gradings. In Section 2 we recall the definition of the enveloping 
algebra $\mathcal{U}(A)$ of a Poisson algebra $A$, and collect some of the 
properties of $\mathcal{U}(A)$ needed for the proof of Theorem \ref{zzthm0.1}. 
In particular, we show that if $A$ is a connected Poisson Hopf algebra, then 
$\mathcal{U}(A)$ is a connected Hopf algebra. Background on the modular 
derivation is contained in Section 3, which also includes the proofs of 
Theorem \ref{zzthm0.1} and Corollary \ref{zzcor0.2}. An example (which is 
not an enveloping algebra of a Lie algebra) to which the main theorem applies 
is given in Example \ref{zzex1.9}, and further examples to illustrate the 
necessity of the hypotheses of the theorem appear in Section 4. Given a Hopf 
algebra, its coproduct, antipode and counit will be 
respectively denoted by $\Delta$, $S$ and $\epsilon$. The definition of the 
Gelfand-Kirillov dimension, denoted $\mathrm{GKdim}$, can be found in \cite{KL}.

\subsection*{Acknowledgments}

The authors would like to thank Ken Goodearl, Xingting Wang and 
Milen Yakimov for many useful conversations and valuable comments 
on the subject. J.J. Zhang was supported by the US National 
Science Foundation (Nos. DMS--1402863 and DMS-1700825).


\section{Definitions and preliminaries}
\label{zzsec1}
\subsection{Poisson algebras}\label{Poisson} 
Recall (from \cite[Definition 1.1]{KS}, for example) that a commutative
associative $\Bbbk$-algebra is a \emph{Poisson algebra} if there is a
$\Bbbk$-bilinear map, the \emph{Poisson bracket} of $A$,
$$\{-,-\}:A \otimes A \longrightarrow A,$$
such that $(A, \{-,-\})$ is a Lie algebra, and the Leibniz rule holds,
namely
$$\{ab,c\} = a\{b,c\} + \{a,c\}b$$
for all $a,b,c \in A$. Given Poisson algebras $A$ and $B$, an algebra
homomorphism $\alpha: A \longrightarrow B$ is a
\emph{Poisson morphism} if $\alpha(\{a,c\}) = \{\alpha(a), \alpha (c) \}$
for all $a,c \in A$. If $A$ and $B$ are Poisson algebras, then by 
\cite[Proposition 1.2.10]{KS} so is $A \otimes B$, with bracket
$$\{a \otimes b, c \otimes d \} = \{a,c\} \otimes bd + ac \otimes \{b,d\}.$$

\begin{definition}
\label{zzdef1.1}
(\cite[pp. 801-802]{Dr2}, \cite[Definitions 3.1.3, 3.1.6]{KS})
\begin{enumerate}
\item[(1)]
A \emph{Poisson Hopf algebra} is a Poisson algebra $A$ which is also
a Hopf algebra, such that $\Delta:A \longrightarrow A \otimes A$ is a Poisson morphism.
\item[(2)]
A \emph{Poisson algebraic group $G$ over $\Bbbk$} is an affine
algebraic $\Bbbk$-group whose coordinate algebra $\Bbbk(G)$ is a
Poisson Hopf algebra; equivalently, the multiplication
$m: G \times G \longrightarrow G$ is a Poisson map.
\end{enumerate}
\end{definition}
In fact,when Definition \ref{zzdef1.1}(1) applies, it is not hard to show that $\epsilon$ is a Poisson morphism and $S$ is a Poisson anti-morphism; see e.g. \cite[p.96]{LWZ1}.

\subsection{Gradings on algebras and coalgebras}
\label{zzsec1.2} 
We fix some terminology to discuss  gradings on three 
operations - the multiplication, the  Poisson bracket, and the 
comultiplication. The graded algebras $A$ under consideration 
in this paper will typically be $\mathbb{N}$-graded; that is, 
$A = \bigoplus_{i \in {\mathbb N}}A(i)$ for $\Bbbk$-subspaces 
$A(i)$, with $A(i)A(j) \subseteq A(i+j)$. We say that $A$ is 
\emph{connected as an algebra} if $A(0) = \Bbbk$.

Suppose that an $\mathbb{N}$- or $\mathbb{Z}$-graded algebra 
$A$ is commutative and Poisson. If there exists an integer $d$ with
$$ \{A(i), A(j)\} \subseteq A(i+j +d) $$
for all $i$ and $j$, we say that the bracket of $A$ has \emph{degree} 
$d$. In this case, $A$ is called a \emph{graded Poisson algebra}, or - equivalently - that $\{-,-\}$ is \emph{homogeneous}.

Recall that the \emph{coradical filtration} of a coalgebra $H$ is
the ascending chain of subcoalgebras $\{H_n : n \geq 0 \}$ of $H$,
with $H_0$ the coradical of $H$ and $H_n$ defined inductively by
$$ H_n \quad := \quad \Delta^{-1}(H \otimes H_0 + H_{n-1} \otimes H) $$
for $n > 0$, \cite[5.2]{Mo}. Then $H = \bigcup_{n \geq 0} H_n$; 
and if $H$ is a pointed Hopf algebra, then $\{H_n\}$ is an algebra 
filtration, \cite[Lemma 5.2.8]{Mo}. We shall say that the Hopf 
algebra $H$ is \emph{connected as a Hopf algebra} if it is connected 
as a coalgebra - that is, if $H_0 = \Bbbk$.

The following lemma is easily proved by induction on $n$.

\begin{lemma}
\label{zzlem1.2} 
\cite[Lemma 7.1]{LWZ1} Let $A$ be a pointed Poisson Hopf algebra 
with coradical filtration $\{A_n\}$. Suppose that $\{g,h\} = 0$ 
for all group-like elements $g,h \in A$. Then, for all $i,j \geq 0$, 
$$ \{A_i, A_j \} \subseteq A_{i+j}. $$
\end{lemma} 

Turning now to gradings on Hopf algebras, the relevant definitions 
are as follows.

\begin{definition}
\label{zzdef1.3}
Let $H$ be a Hopf algebra.
\begin{enumerate}
\item[(1)]
$H$ is a \emph{graded Hopf algebra} if it is simultaneously
$\mathbb{N}$-graded as an algebra and as a coalgebra - that is,
$H = \bigoplus_{i \geq 0}H(i)$ for some vector subspaces $H(i)$,
with $H(i)H(j) \subseteq H(i+j)$, $S(H(i)) \subseteq H(i)$ and
$$\Delta (H(i)) \subseteq 
\bigoplus_{\ell \geq 0} H(\ell) \otimes H(i - \ell)$$
for all $i,j \geq 0$.
\item[(2)]
$H$ is a \emph{coradically graded Hopf algebra} if it is a graded
Hopf algebra with $H_i = \bigoplus_{j \leq i} H(j)$ for all $i \geq 0$.
\end{enumerate}
\end{definition}

When working with graded Hopf algebras it is frequently useful
to make use of the adjustment permitted by the following lemma,
a slight improvement of \cite[Lemma 2.1(2)]{BGZ}.

\begin{lemma}
\label{zzlem1.4}
Let $H$ be a Hopf algebra which is a connected graded algebra,
$H = \bigoplus_{i \geq 0} H(i)$.
\begin{enumerate}
\item[(i)]
There is an algebra grading of $H$, $H = \bigoplus_{i \geq 0} B(i)$,
such that that $\ker \epsilon = \bigoplus_{i \geq 1} B(i)$.
\item[(ii)]
Suppose that $H$ is a graded Hopf algebra through the given
grading $H= \bigoplus_{i \geq 0} H(i)$. Then
$\ker \epsilon = \bigoplus_{i \geq 1} H(i)$.
\end{enumerate}
\end{lemma}
\begin{proof}
(i) This is \cite[Lemma 2.1(2)]{BGZ}.

(ii) Clearly it is enough to show that $H(i) \subseteq \ker \epsilon$
for all $i \geq 1$. So, let $i \geq 1$ and let $x \in H(i)$. We show that
$\epsilon (x) = 0$ by induction on $i$. Since $H$ is a graded coalgebra
and $H(0) = k$,
\begin{equation}
\label{E1.4.1} \tag{E1.4.1}
\Delta (x) \quad = \quad x_1 \otimes 1 + 1 \otimes x_2 + y,
\end{equation}
where $y \in \bigoplus_{j=1}^{i-1}H(j) \otimes H(i-j)$. By induction,
$y = \sum y_1 \otimes y_2$ with all terms $y_1, y_2 \in \ker \epsilon$.
Apply $\mu\circ (\mathrm{Id} \otimes \epsilon)$ to \eqref{E1.4.1},
where $\mu:H \otimes H \rightarrow H$ is multiplication, yielding
$$ x \quad = \quad x_1 + \epsilon (x_2). $$
Since $x$ and $x_1$ are in $H(i)$ while $\epsilon (x_2) \in H(0)$, we
deduce that $\epsilon (x_2) = 0$ and $x_1 = x$. Similarly, using
instead  $\mu\circ (\epsilon \otimes \mathrm{Id})$, it follows that
$\epsilon (x_1) = 0$. Thus $x \in \ker \epsilon$ as required.
\end{proof}

Coradically graded Hopf algebras occur very naturally in the study
of pointed Hopf algebras. The first part of following result is
recorded for example in \cite[Definition 1.13]{AS}; the second 
part is proved in \cite[Theorem 6.9]{Zh}.

\begin{proposition}
\label{zzlem1.5}
Let $H$ be a pointed Hopf algebra with coradical filtration
$\mathcal{C} = \{H_n \}$. Set $H(0) := H_0$ and
$H(i) := H_i/H_{i-1}$ for $i \geq 1$, and let
$$\mathrm{gr}_{\mathcal{C}}H \quad = \quad  \bigoplus_{i \geq 0} H(i)$$
be the associated graded algebra, which is a Hopf algebra with the
operations induced from $H$. 
\begin{enumerate}
\item[(i)] $\mathrm{gr}_{\mathcal{C}}H$
is a coradically graded Hopf algebra.
\item[(ii)] If $H$ is a connected Hopf algebra of finite 
GK-diemension $n$, then $\mathrm{gr}_{\mathcal{C}}H$ is a 
commutative polynomial algebra in $n$ variables.
\end{enumerate}
\end{proposition}

Let $H$ be a graded Hopf algebra. Then it is straightforward 
to show that, keeping the assigned grading,
\begin{equation}
\label{E1.5.1}\tag{E1.5.1} 
H \textit{ a connected graded algebra } \Longrightarrow H 
\textit{ a connected Hopf algebra.}
\end{equation}
To prove \eqref{E1.5.1}, note first that the coradical $H_0$ 
is graded, and then prove that $H_0 \cap H(i) = \{0\}$ for 
$i \geq 1$ by induction on $i$. However, the reverse 
implication is false even when $H$ is commutative, as 
explained in Theorem \ref{zzthm1.6} below.

\subsection{Gradings of commutative Hopf algebras}
\label{zzseec1.3} 
Our focus in this paper is on polynomial algebras and their 
deformations. The possible Hopf structures in this case are 
as described in the following theorem. Details concerning 
part (i) can be found at \cite[Theorem 0.1]{BGZ}. Its 
non-trivial content is due to Serre for $(1)\Longrightarrow (3)$, 
(see e.g. \cite[Theorem 6.2.2]{Be}), and to Lazard \cite{La} for 
$(3)\Longrightarrow(4)$. Note that a unipotent algebraic 
$\Bbbk$-group $U$ is \emph{Carnot} if its Lie algebra 
$\mathfrak{u}$ is $\mathbb{N}$-graded, 
$\mathfrak{u} = \bigoplus_{i=1}^t \mathfrak{u}_i$, with 
$\mathfrak{u} = \langle \mathfrak{u}_1 \rangle$. 
See for example \cite{Co} for further details and references.

\begin{theorem}
\label{zzthm1.6}
Let $H$ be an affine commutative Hopf $\Bbbk$-algebra. 
\begin{itemize}
\item[(i)] 
Then the following are equivalent.
\begin{enumerate}
\item[(1)] 
$H$ is connected graded as an algebra.
\item[(2)] 
$H$ is connected as a Hopf algebra.
\item[(3)] 
$H$ is a polynomial algebra of finite rank.
\item[(4)] 
$H$ is the coordinate algebra of a unipotent algebraic $\Bbbk$-group $U$.
\end{enumerate}
\item[(ii)] 
Let $H$ satisfy the equivalent hypotheses of (i). 
Then the following are equivalent.
\begin{enumerate}
\item[(5)] 
$H$ is coradically graded.
\item[(6)] 
$U$ is a Carnot group.
\end{enumerate}
\end{itemize}
\end{theorem}

\begin{proof}  
(ii)(5)$\Longrightarrow$(6):  
Suppose that $H$ is coradically graded, with, in view of (i),
$$ H = k(U) = k[X_1, \ldots , X_n] = k \oplus \bigoplus_{i \geq 1}H(i),$$
for some $n \geq 1$. Moreover, setting $\mathfrak{u}$ to be the Lie 
algebra of $U$ we have
$$ \mathcal{U}(\mathfrak{u}) \cong \bigoplus_{i \geq 0}H(i)^{\ast}, $$
the graded dual of $H$, by \cite[Proposition 5.5(1),(3)]{BG}, and 
$\mathcal{U}(\mathfrak{u})$ is generated by 
$\mathfrak{u}_1 = H(1)^{\ast}$, by \cite[Lemma 5.5]{AS}, see also 
\cite[Proposition 5.5(4)]{BG}. Thus $\mathfrak{u}$ is a graded 
Lie algebra generated in degree 1. Hence $\mathfrak{u}$ and, 
equivalently, $U$, are Carnot.

(6)$\Longrightarrow$(5): 
Suppose that $H = k(U)$ with $U$ a Carnot group. Consider the 
associated graded algebra $\mathrm{gr}_{\mathcal{C}}H$ with 
respect to the coradical filtration $\mathcal{C} = \{H_n\}$ of 
$H$. By Lemma \ref{zzlem1.5} this is a connected and coradically 
graded Hopf algebra, and it is easy to see - for example, it 
follows by considering the graded dual, as in the proof of 
(5)$\Longrightarrow$(6) - that  $\mathrm{gr}_{\mathcal{C}}H$ is 
the coordinate algebra of the associated graded (Carnot) 
unipotent group $\mathrm{gr}(U)$ of $U$. But $U$ is Carnot, 
so $U = \mathrm{gr}(U)$, and hence $H$ is coradically graded. 
\end{proof}

It is natural to synthesise the concepts defined in subsection 
\ref{zzsec1.2} for the Poisson and Hopf categories, thus 
arriving at graded versions of the concepts introduced by 
Drinfeld, see Definition \ref{zzdef1.1}, as follows.

\begin{definition}
\label{zzdef1.7}
Let $n \geq 1$, let $d_1, \ldots , d_n$ be non-negative integers, and
let $A = k[x_1, \dots , x_n]$, graded so that $x_i$ is homogeneous
of degree $d_i$ for $i=1, \ldots , n$. Suppose that $\{-,-\}$ is a Poisson
bracket on $A$.
\begin{enumerate}
\item[(1)]
If $A$ is a Poisson Hopf algebra, simultaneously graded both as
a Hopf algebra and a Poisson algebra, then $A$ is called a
\emph{graded Poisson Hopf algebra}.
\item[(2)]
If in (1) the Poisson bracket has degree $d$, as defined in 
Section \ref{zzsec1.2},  then $A$ is a
\emph{graded Poisson Hopf algebra of degree $d$}.
\item[(3)]
If in (1) [respectively, (2)] $A$ is coradically graded, then $A$ is a
\emph{coradically graded Poisson Hopf algebra}
[respectively, \emph{of degree $d$}].
\end{enumerate}
\end{definition}

Here are important mechanisms for the construction of coradically 
graded Poisson Hopf algebras. The first is an immediate corollary 
of Lemma \ref{zzlem1.2} and Proposition \ref{zzlem1.5}(i), and 
the second and third follow from Proposition \ref{zzlem1.5}(ii), 
using the commutator bracket on $H$ to induce a Poisson bracket 
on $\mathrm{gr}_{\mathcal{C}}H$. Note that (ii) is the more 
familiar case $t = 1$ of (iii). The proof of the general case 
is straightforward, and is left to the reader. 

\begin{corollary}
\label{zzcor1.8}
\begin{enumerate}
\item[(i)] 
Let $A$ be a pointed Poisson Hopf algebra with coradical 
filtration $\mathcal{C} = \{A_n\}$, and suppose that 
$\{g,h\} = 0$ for all group-like elements $g,h, \in A$. Then 
$\mathrm{gr}_{\mathcal{C}}A$ carries a Poisson bracket 
induced from the bracket on $A$, and as such it is 
coradically graded Poisson Hopf algebra of degree 0. 
\item[(ii)] 
Let $H$ be a connected Hopf algebra of finite GK-dimension 
$n$ and coradical filtration $\mathcal{C}$. Then 
$\mathrm{gr}_{\mathcal{C}}H = \Bbbk [x_1, \ldots , x_n]$ 
is a coradically graded Poisson Hopf algebra of degree $-1$.
\item[(iii)] 
In the setting of part {\rm{(ii)}}, let 
$x_i \in \mathrm{gr}_{\mathcal{C}}H(m_i)$ where 
$m_i \in \mathbb{N}$, and choose 
$y_1, \ldots , y_n \in H$ 
be such that $\mathrm{gr}_{\mathcal{C}}y_i = x_i$, for 
$i = 1, \ldots , n$. Let $t \in \mathbb{Z}$  
be maximal such that $[y_i,y_j] \in H_{m_i + m_j -t}$ for 
all $i$ and $j$, (so that $t \geq1$ by Proposition \ref{zzlem1.5}(ii)). 
Taking images of the brackets $[y_i,y_j]$ 
in the space $H_{m_i + m_j - t}/H_{m_i + m_j -t -1} \subseteq \mathrm{gr}_{\mathcal{C}}H$ 
yields a coradically graded 
Poisson Hopf algebra structure of degree $d = -t$ on 
$\Bbbk [x_1, \ldots , x_n]$. This Poisson bracket is trivial if and only if $H$ is commutative.
\end{enumerate}
\end{corollary}

Similar terminology to that in Definition \ref{zzdef1.7} 
can be introduced in an obvious way for a
commutative $\mathbb{N}$-graded bialgebra $A$. 

We conclude this section by giving  an example of a Poisson Hopf
algebra to which Theorem \ref{zzthm0.1} applies. It should be 
observed that this example illustrates the fact that a 
Poisson Hopf polynomial algebra may be a graded Poisson 
Hopf algebra with respect to more than one grading, and 
that the hypotheses of the main theorem may apply in a 
proper subset of the possible cases.

\begin{example}
\label{zzex1.9}
Let $H$ be the connected Hopf algebra of GK-dimension 5 constructed 
in \cite[Theorem 5.6]{BGZ}. That is, 
$H = k\langle \hat{a}, \hat{b}, \hat{c}, \hat{z}, \hat{w}\rangle$ 
with relations given by setting all commutators of the generators 
equal to 0, except for $[\hat{a}, \hat{b}] = \hat{c}$ and 
$[\hat{z}, \hat{w}] = \frac{1}{3}\hat{c}^3.$ Here, $\hat{a}, \hat{b}, \hat{c}$ 
are primitive,
    $$ \Delta (\hat{z}) = 1 \otimes \hat{z} + \hat{z} \otimes 1 
		+ \hat{a} \otimes \hat{c} + \hat{c} \otimes \hat{a}.$$
and
    $$ \Delta (\hat{w}) = 1 \otimes \hat{w} + \hat{w} \otimes 1 
		+ \hat{b} \otimes \hat{c} + \hat{c} \otimes \hat{b}.$$
It is shown in \cite{BGZ} that $H$ is not isomorphic as an algebra 
to $U(L)$ for any Lie algebra $L$. It is easy to confirm that
    $$ H_1 = k + k\hat{a} + k\hat{b} + k\hat{c} \; \textit{ and } 
		\; H_2 = H_1 + H_1^2 + k\hat{z} + k\hat{w}, $$
so that $A := \mathrm{gr}_{\mathcal{C}}H = k[a,b,c,z,w]$, with 
$a,b$ and $c$ having degree 1, $z$ and $w$ degree 2. Thus, by 
Corollary \ref{zzcor1.8}, $A$ is a coradically graded Poisson 
Hopf algebra of degree $-1$, with
    $$ \{a,b\} = c, \qquad \qquad \{z,w\} = \frac{1}{3}c^3, $$
and all other Poisson brackets of the generators equal to 0.

Another way of understanding $A$ is the following.
Let $T$ be the subgroup of $SL(4,k)$ with $1s$ on the diagonal 
and $0s$ below the diagonal, and let $U = T/Z$, where $Z = Z(T)$, 
which is the subgroup of $T$ with all off-diagonal entries equal 
to 0 except for the $(1,4)$-entry. Then it is easy to check that 
$A \cong \mathcal{O}(U)$; indeed, with the obvious notation, one 
takes $a = X_{12}, b = X_{34}, c = X_{23}, z = X_{13} - S(X_{13}) 
= 2X_{13} - X_{12}X_{23}, w = X_{24} - S(X_{24}) = 2X_{24} - X_{34}X_{23}$.

Now we set
$$ \deg a =\deg b= 1; \qquad \deg c = 2; \qquad \deg z = \deg w = 3. $$
Then one checks that $A$ is a graded Poisson Hopf algebra of Poisson 
degree 0; but of course $A$ is no longer coradically graded. 
Nevertheless, the main theorem, Theorem \ref{zzthm0.1}, still 
applies, and we conclude that $A$ is unimodular. Another way of 
checking the unimodularity of $A$ is to use \eqref{E3.0.1} below.
\end{example}


\section{The enveloping algebra of a Poisson algebra}
\label{zzsec2}

\subsection{Definition of $\mathcal{U}(A)$}
\label{zzsec2.1} 
The proof of Theorem \ref{zzthm0.1} is carried out by passing 
to the enveloping algebra $\mathcal{U}(A)$ of the Poisson 
algebra $A$, whose definition we recall from \cite{Oh2}. Let 
$A$ be a Poisson algebra, and let $m_A=\{ m_a\mid a\in A\}$
and $h_A=\{h_a \mid a\in A\}$ be two copies of the vector
space $A$ endowed with two $\Bbbk$-linear isomorphisms
$m: A \to m_A:a\mapsto m_a$ and $h :A \to h_A : a\mapsto h_a$. 
Then the universal enveloping algebra
$\mathcal{U}(A)$ is an associative algebra over $\Bbbk$, with 
an identity $1$, generated by $m_A$ and $h_A$ with relations
\begin{align}
m_{xy} &= m_x m_y, \label{E2.0.1}\tag{E2.0.1}\\
h_{\{x,y\}} &= h_x h_y - h_y h_x, \label{E2.0.2}\tag{E2.0.2}\\
h_{xy} &= m_yh_x + m_xh_y, \label{E2.0.3}\tag{E2.0.3}\\
m_{\{x,y\}} &= h_xm_y - m_yh_x = [h_x,m_y], \label{E2.0.4}\tag{E2.0.4}\\
m_1 &= 1. \label{E2.0.5}\tag{E2.0.5}
\end{align}

Given a commutative $\Bbbk$-algebra $A$, an $A$-module and 
Lie algebra $L$, and an $A$-module and Lie algebra map $\rho$ 
from $L$ to $\mathrm{Der}_{\Bbbk}A$ satisfying a compatibility 
condition, Rinehart  \cite{Ri} defined in 1963 a certain 
associative algebra which he denoted $V(A,L)$ and which is 
nowadays called the 
\emph{the enveloping algebra of the Lie-Rinehart algebra} of 
$A$, $L$ and the \emph{anchor map} $\rho$. Huebschmann showed 
in 1990 \cite[Theorem 3.8]{Hu} that one can construct such an 
enveloping algebra starting from any Poisson algebra $(A,  \{-,-\})$, 
taking $L$ to be the $A$-module $\Omega_A$ of K{\"a}hler differentials 
of $A$, with $\rho(da) = \{a,-\}$ for $a \in A$. In fact it follows 
from earlier work of 
Weinstein \emph{et al} \cite{We1}, \cite{CDW},  that there is an 
algebra isomorphism 
$$ \mathcal{U}(A) \; \cong \; V(A,\Omega_A),$$ which is the 
identity on $A$; a detailed account of this isomorphism can 
be found as \cite[Proposition 5.7]{LWZ1}.

\subsection{Properties of $\mathcal{U}(A)$}
\label{zzsec2.2} 
Most of the following facts about $\mathcal{U}(A)$ which we need 
are already in the literature, or are easy consequences of known results.

\begin{proposition}
\label{zzpro2.1} 
Let $A$ be an affine Poisson $\Bbbk$-algebra, 
$A = \Bbbk \langle x_1, \ldots , , x_n \rangle$.  
Let $\mathcal{U}(A)$ be the Poisson enveloping algebra of $A$.
\begin{enumerate}
\item[(i)] 
$\mathcal{U}(A)$ is an affine $\Bbbk$-algebra, 
generated by $\{m_{x_i}, h_{x_i} : 1 \leq i,j \leq n\}$.
\item[(ii)] 
(Here we abuse notation slightly by simply writing $a$ for the image $m_a$ of 
$a \in A$  in $\mathcal{U}(A)$.) When $\mathcal{U}(A)$ is $\mathbb{Z}_{\geq 0}$-filtered 
by the filtration $\mathcal{F}$ obtained  by assigning 
$\mathrm{deg} a = 0$ for $a \in A$ and $\mathrm{deg} 
h_{x_i} = 1$ for $i = 1, \ldots , n$, $\mathrm{gr}_{\mathcal{F}}(A)$ 
is a commutative affine $\Bbbk$-algebra with a generating 
set of cardinality at most $2n$.
\item[(iii)] 
Suppose that $A$ is regular, with module of K{\"a}hler 
differentials $\Omega (A)$. Then
$$\mathrm{gr}_{\mathcal{F}}(A) \; \cong \; \mathrm{Sym}_A(\Omega_A),$$
the symmetric algebra of $\Omega_A$ over $A$.
\item[(iv)] Suppose that $A$ is regular of global dimension $t$. 
Then $\mathrm{GKdim}(\mathcal{U}(A)) = 2t.$
\end{enumerate}
\end{proposition}

\begin{proof}
(i), (ii): These are easily deduced from the defining relations 
for $\mathcal{U}(A)$, (E2.0.1)-(E2.0.5). 

(iii) This is the Rinehart's PBW theorem \cite[Theorem 3.1]{Ri}, 
since $\Omega_A$ is $A$-projective when $A$ is regular.

(iv) It follows from \cite[Sec. 1.4, Corollary]{MS} that
$$ \mathrm{GKdim}(A) \; = \; \mathrm{GKdim}(\mathrm{gr}_{\mathcal{F}}(A)),$$
since $\mathrm{gr}_{\mathcal{F}}(A)$ is an affine commutative 
$\Bbbk$-algebra by (3). However the GK-dimension of 
$\mathrm{gr}_{\mathcal{F}}(A)$ equals its Krull dimension. 
Let $\mathfrak{m}$ be a maximal ideal of  $\mathrm{gr}_{\mathcal{F}}(A)$, so that 
$\mathfrak{m}' := \mathfrak{m} \cap A$ is a maximal ideal of $A$. 
The localisation of  $\mathrm{gr}_{\mathcal{F}}(A)$  at 
$A \setminus \mathfrak{m}'$ is  
$ \mathrm{Sym}_{A_{\mathfrak{m}'}} (\Omega_{A_{\mathfrak{m}'}})$. 
This algebra is thus a polynomial algebra in $t'$ variables over 
the $t'-$dimensional regular local ring $A_{\mathfrak{m}'}$, 
where $t' = \mathrm{height}(\mathfrak{m}') \leq t$.  Thus all of 
its maximal ideals, and in particular 
$\mathfrak{m} \mathrm{Sym}_{A_{\mathfrak{m}'}} (\Omega_{A_{\mathfrak{m}'}})$, 
have height $2t'$. Since the maximum value of $t'$ is 
$t$, $\mathrm{Sym}_{A} (\Omega_{A})$  has Krull dimension $2t$ as required.
\end{proof}

\subsection{Gradings on $\mathcal{U}(A)$}
\label{zzsec2.3}
When $A$ is graded there are straightforward ways to extend 
the grading to $\mathcal{U}(A)$. The following result summarises 
what we will need.

\begin{proposition} 
\label{zzpro2.2} 
Let $A$ be an affine Poisson $\Bbbk$-algebra, 
$A = \Bbbk \langle x_1, \ldots , , x_n \rangle$.  Let 
$\mathcal{U}(A)$ be the Poisson enveloping algebra of $A$.
\begin{enumerate}
\item[(i)]
Suppose that $A$ is ${\mathbb Z}$-graded and $\{-,-\}$ is homogeneous of
degree $d$. Then $\mathcal{U}(A)$ is ${\mathbb Z}$-graded with
$\deg m_x=\deg x$ and $\deg h_x=\deg x+d$ for all homogeneous
elements $x$ in $A$.
\item[(ii)] 
If, in (i), $A$ is connected $\mathbb{N}$-graded with $d \geq 0$, 
then $\mathcal{U}(A)$ is also connected $\mathbb{N}$-graded.
\item[(iii)]
Suppose that $A$ is a  connected $\mathbb{N}$-graded algebra, generated in degree 1, and that
$\{-,-\}$ is homogeneous of degree $d\geq 0$. Then
$\mathcal{U}(A)$ is a connected $\mathbb{N}$-graded algebra, and is 
minimally generated by $m_{A_1}$ and $h_{A_1}$.
\end{enumerate}
\end{proposition}

\begin{proof}
(i),(ii)  First note that $m: A\to m_A$ and
$h: A[d]\to h_A$ are
graded $\Bbbk$-linear maps. It is clear that $\mathcal{U}(A)$ is 
generated by
homogeneous elements in $m_A$ and $h_A$ since both $m$ and $h$ are
$\Bbbk$-linear. The relations of $\mathcal{U}(A)$ given in 
\eqref{E2.0.1}-\eqref{E2.0.5}
are homogeneous. Therefore $\mathcal{U}(A)$ is naturally 
${\mathbb Z}$-graded. Finally, (ii) is an immediate consequence of (i).

(iii) In this case we can assume that the generators $x_i$ are 
linearly independent and homogeneous of degree 1, so 
$\deg m_{x_i}=1$ and $\deg h_{x_i}=1+d\geq 1$ for all 
$i = 1, \ldots , n$. Therefore
$\mathcal{U}(A)$ is connected graded, by (ii). Since 
$\deg m_{x_i}=1$ for all $i$, $\{m_{x_i} : 1 \leq i \leq n \}$ is a linearly
independent subset of a minimal homogeneous generating set.
Let $\fm$ be the maximal graded ideal of $\mathcal{U}(A)$. It is 
clear that $\fm/\fm^2$ is spanned by 
$\{m_{x_i} : 1 \leq i \leq n \} \cup \{h_{x_i} : 1 \leq i \leq n \}$. 
Let $x \in A(1)$. Since
$\deg h_x=1+d$, if $h_x$ is in $\fm^2$, then it is generated by
elements of small degrees, namely by $m_{A(1)}$. However, any homogeneous
relation involving $h_x$ (see \eqref{E2.0.2}, \eqref{E2.0.3} and \eqref{E2.0.4})
has degree at least $d+2$. Therefore $h_x\in \fm^2$ does not follow
from any relations of $\mathcal{U}(A)$, a contradiction. Thus 
$\{h_{x_i} : 1 \leq i \leq n \}$ maps to a linearly independent subset 
of $\fm/\fm^2$. The same argument shows that 
$(m_{A(1)} + \fm^2/\fm^2) \cap( h_{A(1)} + \fm^2/\fm^2) = 0$. Hence 
$\mathcal{U}(A)$ is minimally generated by $m_{A(1)} \cup h_{A(1)}$.
\end{proof}

\subsection{Bialgebra structure on $\mathcal{U}(A)$}
\label{zzsec2.4}
We shall need the following results of \cite{Oh1} and \cite{LWZ1}. 
The proof of the claim in \cite{Oh1} that a Poisson Hopf structure 
on an algebra $A$ always induces a structure of Hopf algebra on 
$\mathcal{U}(A)$ appears not to be completely clear as regards the 
extension of the  antipode from $A$ to $\mathcal{U}(A)$, so we address 
that aspect separately, for the cases of concern to us, in what follows.

\begin{theorem}
\label{zzpro2.3} 
Let $A$ be a Poisson Hopf algebra.
\begin{enumerate}
\item[(i)] $\mathcal{U}(A)$ admits a bialgebra structure with $A$ as a sub-bialgebra.
\item[(ii)] The coradical of  $\mathcal{U}(A)$ is the coradical of $A$.
\item[(iii)] If $A$ is pointed, then $\mathcal{U}(A)$ is a Hopf algebra 
with sub-Hopf algebra $A$, with $G(\mathcal{U}(A)) = G(A)$. 
\item[(iv)] If $A$ is a connected Hopf algebra, then so is $\mathcal{U}(A)$.
\end{enumerate}
\end{theorem} 
\begin{proof} (i) See \cite[Theorem 10]{Oh1}.

(ii) This is \cite[Proposition 6.6]{LWZ1}.

(iii) Suppose that $A$ is pointed. Then (ii) shows that $\mathcal{U}(A)$ 
is also pointed, with $G(\mathcal{U}(A)) = G(A)$, so that every 
grouplike element of $\mathcal{U}(A)$ is invertible. It follows 
from \cite[Theorem 1]{Ra} that the bialgebra $\mathcal{U}(A)$ is a Hopf algebra.

(iv) This is a special case of (iii).
\end{proof}

\section{Proof of theorem \ref{zzthm0.1}}
\label{zzec3}

\subsection{The modular derivation.}
\label{zzsubsect3.1}
We refer to \cite{Do, LWW} for further details about 
several important invariants of a Poisson algebra. There is a 
geometric notion of
unimodular Poisson manifold introduced by Weinstein \cite{We2}. 
In algebraic terms, let $A$ be a smooth Poisson algebra with 
trivial canonical 
bundle - for instance in our case $A$ will be a polynomial 
$\Bbbk$-algebra in $n$ variables. A \emph{Poisson derivation} 
$d$ of $A$ is a derivation in the usual sense with the 
additional property that 
$$d(\{a,b\}) = \{d(a),b\} + \{a,d(b)\}$$ for all $a,b \in A$. 
Then the {\it modular derivation} of $A$ (with respect to a 
given volume form) is a certain Poisson derivation $\delta$, 
defined in \cite[Definition 2.3]{LWW}. Now the 
{\it modular class} of $A$ is the class of $\delta$ in the space of Poisson derivations 
modulo the subgroup of \emph{log-Hamiltonian} derivations, see both
\cite[p. 208]{Do} and \cite[Section 2.2]{LWW}.
If the modular class of $A$ is trivial, then $A$ is called {\it unimodular}.
If $A$ is the polynomial ring (or more generally, if the only 
invertible elements in $A$ are scalars),
then there are no non-zero log-Hamiltonian derivations. In this case, $A$
is unimodular precisely when the modular derivation is zero.

When $A$ is the polynomial ring $\Bbbk[x_1,\cdots,x_n]$ (and 
the volume form can be taken 
to be $\eta:= d x_1 \wedge dx_2 \wedge \cdots \wedge dx_n$),
the modular derivation $\delta$ of $A$ is given in \cite[Lemma 2.4]{LWW}:
namely,
\begin{equation}
\label{E3.0.1}\tag{E3.0.1}
\delta(f)(=\delta_{\eta}(f))=\sum_{j=1}^n \frac{\partial\{ f,x_j\}}{\partial x_j}
\end{equation}
for all $f\in A$.

We refer to \cite{RRZ} for some basic definitions concerning Calabi-Yau algebras
(or CY algebras for short), skew Calabi-Yau algebras and the Nakayama
automorphism. There is a close connection between modular
derivations and Nakayama automorphisms via deformation
theory, see \cite[Theorem 2]{Do}. A further such result, in this case
connecting  the modular derivation of a Poisson algebra with the
Nakayama automorphism of its Poisson enveloping algebra, is the 
following. We shall apply the result below in the case where $A$ 
is a polynomial algebra. In its statement, one has to interpret 
the expression ``the ... automorphism ...  of $\mathcal{U}(A)$ 
... is  ... $2\delta$'' in the following manner. Recall that 
$\mathcal{U}(A)$ is generated by $\{m_a, h_a : a \in A\}$. So, 
if $\tau$ is a Poisson derivation of $A$, we define the 
corresponding map on $\mathcal{U}(A)$ by $\tau (m_a) = m_a$ 
and $\tau(h_a) = h_a + m_{\tau (a)}$. For details, see 
\cite[Lemma 2.2]{LWZ2}.

\begin{proposition}
\cite[Corollary 5.6, Proposition 1.12, Theorem 5.8 and Remark 5.9]{LWZ2}
\label{zzlem3.1}
Let $A$ be a CY Poisson algebra; that is, by definition, $A$ is an affine 
smooth Poisson algebra of finite global dimension $n$, with trivial 
canonical bundle. Then $\mathcal{U}(A)$ is skew
CY of dimension $2n$. Moreover, the Nakayama automorphism of $\mathcal{U}(A)$
is given by  $2\delta$, where $\delta$ is the modular derivation
of $A$. As a consequence, if $A$ is a polynomial ring, then $A$ is
unimodular if and only if  the modular derivation is zero, and  if and only
if $\mathcal{U}(A)$ is CY.
\end{proposition}

\subsection{Proof of Theorem \ref{zzthm0.1}.} 
Theorem \ref{zzthm0.1} follows easily from the above proposition 
and \cite[Theorem 0.2]{BGZ}.

\begin{proof} Since $A$ is a commutative affine connected graded 
algebra of finite global dimension, it is a polynomial algebra, 
by \cite[Theorem 6.2]{Be}. Its coalgebra structure is 
thus also connected, by Theorem \ref{zzthm1.6}.
By the hypothesis on the grading of the Poisson bracket and by 
Proposition \ref{zzpro2.2}(ii), $\mathcal{U}(A)$ is a 
connected $\mathbb{N}$-graded algebra, and 
has finite GKdimension by Proposition \ref{zzpro2.1}(iv).
By Theorem \ref{zzpro2.3}(iv), $\mathcal{U}(A)$ is a 
connected Hopf algebra which is moreover connected 
graded as an algebra. Applying \cite[Theorem 0.2]{BGZ}, 
$\mathcal{U}(A)$ is Calabi-Yau. It follows
from Lemma \ref{zzlem3.1} that the modular derivation 
$\delta$ of $A$ is zero - that is, $A$ is unimodular.
\end{proof}

Corollary \ref{zzcor0.2} is an immediate consequence 
of Theorem \ref{zzthm0.1} and 
\cite[Theorem 3.5 and Remark 3.6]{LWW}
(or \cite[Theorem 4.3]{LWZ2}).

\section{Examples and counterexamples}
\label{zzsec4}

Theorem \ref{zzthm0.1} does not remain true for Poisson bialgebras
as the following example demonstrates.

\begin{example}
\label{zzex4.1}
Let $A$ be the polynomial algebra $\mathbb{C}[x,y]$, let $i$ be a 
non-negative integer, and consider the bialgebra structure on $A$ given by
$$ \Delta (y) = y \otimes y, \quad \Delta (x) = x \otimes 1 + y^i \otimes x,$$
with $\epsilon (x) = 0$ and $\epsilon (y) = 1$.Then, setting $\{x,y\} = xy$, 
it is straightforward to check that $A$ is a Poisson bialgebra. Moreover, 
$A$ is a connected graded algebra, with the Poisson bracket homogeneous of 
degree 0. However, $A$ is \emph{not} unimodular - one calculates from 
\eqref{E3.0.1} that the modular derivation $\delta$ of $A$ is given by 
$\delta (x) = x$ and $\delta (y) = -y$.
\end{example}

\begin{example} \cite[Example 2.8(2)]{LWW}
\label{zzex4.2}
$\textbf{The Kostant-Souriau Poisson bracket.}$  
The familiar Poisson bracket induced by the bracket 
of a Lie algebra $\mathfrak{g} = \sum_{i=1}^n \Bbbk x_i$ 
on its symmetric algebra $S(\mathfrak{g})$ shows that 
the hypothesis in Theorem \ref{zzthm0.1} that the 
Poisson bracket is homogeneous of degree $d \geq 0$ is necessary.

For, with its polynomial generators $x_i$ all assigned 
degree 1, $S(\mathfrak{g})$ is a 
coradically graded Poisson Hopf algebra of degree $d= -1$.
Using \eqref{E3.0.1} one easily calculates that the modular 
derivation is given by 
$$ \delta (x_i) = \mathrm{tr}(\mathrm{ad}(x_i))    \quad 1 \leq i \leq n,$$
where $\mathrm{ad}$ denotes the adjoint representation of $\mathfrak{g}$.
\end{example}

\begin{example}
\label{zzex4.3}
$\textbf{Connected Hopf algebras of small GK-dimension.}$ 
The connected Hopf $\Bbbk$-algebras of GK-dimensions 3 and 4 are 
determined in \cite{Zh} and \cite{WZZ} respectively. Applied to 
these algebras, the recipe of Corollary \ref{zzcor1.8}(ii) and 
(iii) yields infinite families of coradically graded Poisson 
Hopf algebra structures  on the polynomial $\Bbbk$-algebras in 
3 and 4 variables. Two of these families, manufactured respectively 
from \cite[Section 7, Example 2]{Zh} and \cite[Example 4.4]{WZZ}, 
are given as \cite[Examples 3.2, 3.3]{LWZ1}. Here are brief details 
of the 3-dimensional examples.

(a) Let  $\lambda, \mu \in k$ and $\alpha \in \{0,1\}$, and let 
$A :=A(\lambda, \mu, \alpha)$ be the family of Hopf $\Bbbk$-algebras 
of GK-dimension 3 defined at \cite[Section 7, Example 2]{Zh}. So 
$A = k\langle X,Y,Z \rangle$, with
$$ [X,Y] = 0, \quad [Z,X] = \lambda X + \alpha Y, \quad [Z, Y] = \mu Y. $$
The space of primitive elements is $kX \oplus kY$, and 
$\Delta (Z) = Z \otimes 1 + 1 \otimes Z + X \otimes Y,$ so that 
$Z \in A_2$. Thus, applying Corollary \ref{zzcor1.8}(iii) and 
using the obvious notation,
$$ \mathrm{gr_{\mathcal{C}}}A(\lambda, \mu, \alpha) = k[x,y,z]; 
\quad \mathrm{deg} x = \mathrm{deg} y = 1; \quad \mathrm{deg} z = 2. $$
Therefore $\mathrm{gr}_{\mathcal{C}}A(\lambda, \mu, \alpha) = \mathcal{O}(U)$ 
is a coradically graded Poisson Hopf algebra of Poisson degree $-2$, with 
$U$ the 3-dimensional Heisenberg group and
$$ \{x,y\} = 0, \quad \{z,x\} = \lambda x + \alpha y, \quad \{z,y\} = \mu y. $$
Using \eqref{E3.0.1}, the modular derivation of the Poisson algebra 
$\mathcal{O}(U)$ is determined by
$$\delta(x)=0,\; \delta(y)=0, \; \delta(z)=\lambda+\mu.$$
By \cite[Proposition 7.9]{Zh}, a complete set of isomorphism classes 
of the Hopf algebras $A(\lambda, \mu, \alpha)$ is given by
\begin{equation}\label{E4.3.1}
\tag{E4.3.1} 
\{(1,0,0),(0,0,0), (0,0,1), (1,1,1), (1,\mu^{\pm 1},0) : \mu \in k^* \}.
\end{equation}
It is not hard to deduce from this result that the same list of parameter 
values (\ref{E4.3.1}) gives the complete list of isomorphism classes of 
the Poisson Hopf algebras $\mathrm{gr}_{\mathcal{C}}A(\lambda, \mu, \alpha)$. 
We sketch an argument. It is enough to show that distinct parameter triples 
from \eqref{E4.3.1} yield distinct Poisson Hopf algebras. Suppose that $\theta$ 
is an isomorphism of Poisson Hopf algebras between two such algebras. Since 
$\theta$ preserves the counits, that is $\theta (\mathfrak{m}) = \mathfrak{n}$ 
say, it induces an isomorphism of Lie algebras 
$(\mathfrak{m}/\mathfrak{m}^2, \{-,-\}) \cong (\mathfrak{n}/\mathfrak{n}^2, \{-,-\})$. 
Moreover, $\theta$ preserves the coalgebra structure, and hence  preserves 
the structures of Lie bialgebras (see \cite[Section 6]{LWZ1}) on 
$(\mathfrak{m}/\mathfrak{m}^2, \{-,-\})$ and $(\mathfrak{n}/\mathfrak{n}^2, \{-,-\})$. 
From this the desired conclusion easily follows.

(b) Let $B(\lambda)$ be the family of Hopf $k$-algebras of GK-dimension 3 
defined at \cite[Section 7, Example 3]{Zh}. So $B(\lambda) = k\langle X,Y,Z \rangle$, 
with $\lambda \in k$ and
$$[X,Y] = Y; \quad [Z,X] = -Z + \lambda Y; \quad [Z,Y] = \frac{1}{2}Y^2. $$
The space of primitive elements is $kX \oplus kY$ and 
$\Delta (Z) = Z \otimes 1 + 1 \otimes Z + X \otimes Y.$ Thus, exactly as 
with (a) we get a coradically graded Poisson Hopf algebra
$$ \mathrm{gr}_{\mathcal{C}}B(\lambda) = k[x,y,z]; 
\quad \mathrm{deg} x = \mathrm{deg} y = 1; \quad \mathrm{deg} z = 2. $$
This time, however, the Poisson degree $d$ is $-1$. We see that  
$\mathrm{gr}_{\mathcal{C}}B(\lambda)$ is a Poisson Hopf algebra with 
the underlying Poisson unipotent group again being the 3-dimensional 
Heisenberg group, with
$$ \{x,y\} = y, \quad \{z,x\} = -z, \quad \{z,y\} = \frac{1}{2} y^2. $$
Using \eqref{E3.0.1}, the modular derivation of this Poisson algebra is
determined by
$$\delta(x)=2,\; \delta(y)=0, \; \delta(z)=y.$$
From \cite[Proposition 7.10]{Zh}, $B(\lambda) \cong B(\mu)$ as Hopf 
algebras if and only if $\lambda = \mu$, whereas it is evident that 
one obtains the \emph{same} Poisson Hopf algebra $\mathrm{gr}_{\mathcal{C}}B(\lambda)$ 
for every value of $\lambda \in k$.
\end{example}

\vspace{0.5cm}


\begin{thebibliography}{50}


\bibitem[AS]{AS} 
N. Andruskiewitsch and H.-J. Schneider,
``Pointed Hopf Algebras'', 
pages 1--68 of \emph{New Directions in Hopf Algebras} MSRI Publications 
{\bf 43}, Cambridge University Press, 2002.

\bibitem[BYZ]{BYZ}
Y.-H. Bao, Y. Ye and J.J. Zhang, 
Restricted Poisson algebras, 
\emph{Pacific J. Math.} {\bf 289} (2017), no. 1, 1--34. 

\bibitem[Be]{Be}
D.J. Benson,
\emph{Polynomial Invariants of Finite Groups},
London Math. Soc. Lecture Notes $(\textbf{190})$, 
cambridge University Press, 1993.


\bibitem[BK]{BK} R. Bezrukavnikov and D. Kaledin, 
Fedosov quantization in positive characteristic, 
\emph{J. Amer. Math. Soc.} {\bf 21}  (2008), 409--438.

\bibitem[BG]{BG} 
K.A. Brown and P. Gilmartin, 
Quantum homogeneous spaces of connected Hopf algebras, 
\emph{J. Algebra} {\bf 454} (2016), 400-432.



\bibitem[BGZ]{BGZ}
K.A. Brown, P. Gilmartin and J.J. Zhang, Connected (graded) Hopf algebras,
preprint (2016), arXiv:1601.06687.

\bibitem[CDW]{CDW}
A. Coste, P. Dazord, A. Weinstein,
Groupoides symplectiques,
Univ. Claude-Bernard, Lyon, 1987, pp 1-62, i-ii.


\bibitem[Co]{Co} 
Y. Cornulier, 
Gradings on Lie algebras, systolic growth, and coHopfian properties of nilpotent groups, 
\emph{Bull. Soc. Math. France} {\bf 144}, 693-744.


\bibitem[Do]{Do}
V.A. Dolgushev,
The Van den Bergh duality and the modular symmetry of a Poisson variety,
\emph{Selecta Math.} (N.S.) {\bf 14} (2009) 199--228.

\bibitem[Dr1]{Dr1}
V.G. Drinfeld,
Hopf algebras and the quantum Yang–Baxter equation, 
Dokl. Akad. Nauk SSSR {\bf 283}:5 (1985), 1060--1064. 
In Russian; translated in Soviet Math. Dokl. {\bf 32}:1 (1985),
254--258.


\bibitem[Dr2]{Dr2}
V.G. Drinfel'd,
Quantum groups,
\emph{Proc. International Congress of Mathematicians}, Berkeley, Ca, USA, 1986.










\bibitem[Hu]{Hu} 
J. Huebschmann, 
Poisson cohomology and quantization, 
\emph{J. Reine Angew. Math.} {\bf 408} (1990), 57--113.


\bibitem[KS]{KS}
L.I. Korogodski and Y.S. Soibelman,
\emph{Algebras of Functions on Quantum Groups: Part 1},
Mathematical Surveys and Monographs {\bf 56}, Amer. Math. Soc. 1998.

\bibitem[KL]{KL}
G. Krause and T.H. Lenagan,
\emph{Growth of Algebras and Gel'and-Kirillov Dimension},
Revised Edition, Graduate Texts in Mathematics {\bf 22},
Amer. Math. Soc. 2000.

\bibitem[LW]{LW}
Q. Lou and Q.-S. Wu,
Co-Poisson structures on polynomial Hopf algebras,
\emph{Sci. China Math.} (accepted for publication), 
preprint (2016), arXiv 1601.042969.

\bibitem[LWZ1]{LWZ1}
J. L\"u, X. Wang and G. Zhuang,
Universal enveloping algebras of Poisson Hopf algebras,
\emph{J. Algebra}, {\bf  426}  (2015), 92--136.

\bibitem[LWZ2]{LWZ2}
J. L\"u, X. Wang and G. Zhuang,
Homological unimodularity and Calabi-Yau condition for Poisson algebras,
\emph{Lett Math Phys}, DOI 10.1007/s11005-017-0967-6,
preprint (2016),  arXiv:1608.00172.




\bibitem[LWW]{LWW}
J. Luo, S.-Q. Wang and Q.-S. Wu,
Twisted Poincar$\acute{\mathrm{e}}$ duality between Poisson homology and Poisson cohomology,
\emph{J. Algebra}, {\bf  442} (2015), 484--505.



\bibitem[La]{La}
M. Lazard, 
Sur la nilpotence de certains groupes algebriques, 
\emph{C. R. Acad. Sci.} Ser. 1 Math.
{\bf 41} (1955), 1687 – 1689.

\bibitem[LY]{LY}
J Levitt and M. Yakimov,
Quantized Weyl algebras at roots of unity,
preprint, (2016), arXiv:1606.02121.

\bibitem[MS]{MS} 
J.C. McConnell and J.T.  Stafford, 
\emph{Gel′fand-Kirillov dimension and associated graded modules}, 
J. Algebra  {\bf 125}  (1989),  no. 1, 197--214. 



\bibitem[Mo]{Mo}
S. Montgomery,
\emph{Hopf Algebras and their Actions on Rings},
CBMS Regional Series in Mathematics {\bf 82}, Providence, RI, 1993.


\bibitem[NTY]{NTY}
B. Nguyen, K. Trampel, and M. Yakimov, 
Noncommutative discriminants via Poisson primes,
preprint, (2016), arXiv:1606.02121.

\bibitem[Oh1]{Oh1}
S.-Q. Oh,
Hopf structure for Poisson enveloping algebras,
\emph{Beitr{\"a}ge Algebra Geom.} {\bf 44} (2) (2003) 567--574.

\bibitem[Oh2]{Oh2}
S.-Q. Oh,
Poisson enveloping algebras,
\emph{Comm. Algebra} {\bf 27} (1999), 2181--2186.

\bibitem[Ra]{Ra}
D.E. Radford,
On bialgebras which are simple Hopf modules,
\emph{Proc. Amer. Math. Soc.} {\bf 80} (1980), 563--568.



\bibitem[Ri]{Ri}
G.S. Rinehart,
Differential forms on general commutative algebras,
\emph{Trans. Amer. Math. Soc.} {\bf 108} (1963), 195--222.


\bibitem[RRZ]{RRZ}
M. Reyes, D. Rogalski, and J. J. Zhang,
Skew Calabi-Yau algebras and homological identities,
\emph{Adv. Math.} \textbf{264} (2014), 308--354.




\bibitem[WWY]{WWY}
C. Walton, X.-T. Wang and M. Yakimov,
The Poisson geometry of the 3-dimensional Sklyanin algebras,
preprint, (2017), arXiv:1704.04975.


\bibitem[WZZ]{WZZ}
D.-G. Wang, J.J. Zhang, G. Zhuang,
Classification of connected Hopf algebras of Gelfand-Kirillov dimension four,
\emph{Trans. Amer. Math. Soc.} {\bf367} (2015), 5597-5632.

\bibitem[We1]{We1}
A. Weinstein,
Symplectic groupoids and Poisson manifolds,
\emph{Bull. Amer. Math. Soc.} {\bf 16} (1987), 101--104.

\bibitem[We2]{We2}
A. Weinstein,
The modular automorphism group of a Poisson manifold,
\emph{J. Geom. Phys.}   {\bf 23} (1997), 379--394.


\bibitem[Zh]{Zh} 
G.-B. Zhuang, 
Properties of pointed and connected Hopf algebas of 
finite Gelfand-Kirillov dimension, 
\emph{J. London Math. Soc.} (2) {\bf 87} (2013), 877-898.

\end{thebibliography}
\end{document}